\documentclass[11pt]{article} 
\usepackage{amssymb, latexsym}
\newtheorem{defin}{}
\newtheorem{saetze}[defin]{}
\newtheorem{lemmas}[defin]{}
\newtheorem{folger}[defin]{}
\newtheorem{bemerk}[defin]{}
\newtheorem{exampl}[defin]{}
\newtheorem{scholi}[defin]{}
\newtheorem{conjec}[defin]{}

\newenvironment{theorem}  {\begin{saetze}\it {\bf Theorem:}}{\end{saetze}}
\newenvironment{scholium} {\begin{scholi}\it {\bf Scholium:}}{\end{scholi}}
\newenvironment{conjecture} {\begin{conjec}\it {\bf Conjecture:}}{\end{conjec}}
\newenvironment{lemma}    {\begin{lemmas}\it {\bf Lemma:}}{\end{lemmas}}

\newenvironment{remark}   {\begin{bemerk}\it {\bf Remark:}}{\end{bemerk}}
\newenvironment{example}  {\begin{exampl}\it {\bf Example:}}{\end{exampl}}
\newenvironment{proof}    {{\it Proof}:}{{\hfill \fillbox \bigskip}}

\newcommand{\fillbox}{\mbox{$\bullet$}}
\newcommand{\ra}{\rightarrow}

\newcommand{\ms}{\mapsto}
\newcommand{\ol}{\overline}
\newcommand{\ti}{\tilde}

\newcommand{\N}{\mathbb N}

\newcommand{\Z}{\mathbb Z}

\newcommand{\RELS}{\mathcal R}
\newcommand{\GAP}{{\sf GAP }}

\newenvironment{items}{\begin{list}{$\alph{item})$}
{\labelwidth18pt \leftmargin18pt \topsep3pt \itemsep1pt \parsep0pt}}
{\end{list}}

\newcommand{\bulit}{\item[$\bullet$]}

\begin{document}

\title{A nilpotent quotient algorithm \\ for $L$-presented groups}
\author{Bettina Eick and Ren\'{e} Hartung\\
        with an appendix by Laurent Bartholdi}
\date{June 21, 2007}
\maketitle

\begin{abstract}
The main part of this paper contains a description of a nilpotent quotient 
algorithm for $L$-presented groups and a report on applications of its 
implementation in the computer algebra system {\sf GAP}. The appendix 
introduces
two new infinite series of $L$-presented groups. Apart from being of 
interest in their own right, these new $L$-presented groups serve as 
examples for applications of the nilpotent quotient algorithm.
\end{abstract}

\section{Introduction}

A very useful group theoretic construction is that of an ascending
HNN-extension: for example, this can have the form $\ol{G} = \langle 
G, t \mid g^t = \varphi(g) \mbox{ for all } g \in G \rangle$ where 
$\varphi$ is an injective endomorphism of $G$. Such a construction 
is often used to embed a non-finitely-presented group $G$ into a 
finitely presented group $\ol{G}$ to obtain a compact description of 
the group $G$.

$L$-presentations have been introduced by Bartholdi \cite{Bar03} based
on work of Lysenok \cite{Lys85}. They provide a framework to describe 
the base group $G$ of an HNN-extension, even though it is not finitely 
presented, and they extend this key idea further. More precisely, an 
$L$-presentation is an expression of the form 
$\langle S \mid Q \mid \Phi \mid R \rangle$,
where $Q$ and $R$ are subsets of the free group $F_S$ and $\Phi$ is 
a set of endomorphisms of $F_S$, and it defines the group
\[ F_S / \langle Q \cup \bigcup_{\varphi \in \Phi^*} \varphi(R) 
         \rangle^{F_S}, \]
where $\Phi^*$ is the monoid generated by $\Phi$. In \cite{Bar03} it is 
proved that a broad class of groups acting on rooted trees admit an 
explicitly constructible finite $L$-presentation. Well-known examples 
of finitely $L$-presented, but not finitely presented groups are the 
Grigorchuk group and the Gupta-Sidki group which both play a role in
the study of the famous Burnside problems.

The central aim of the main part of this paper is to describe a nilpotent 
quotient algorithm for finitely $L$-presented groups. This takes as input a
finitely $L$-presented group $G$ and a positive integer $n$ and determines 
a nilpotent presentation for the class-$n$ quotient $G / \gamma_{n+1}(G)$. 
Thus this algorithm can also determine the abelian invariants of $G$ and the 
largest nilpotent quotient of $G$ if it exists. The algorithm described here
generalises the nilpotent quotient algorithm for finitely presented groups 
by Nickel \cite{Nic95}. 

An implementation of the algorithm described here is available in the 
{\sf NQL} package \cite{NQL} of the computer algebra system \GAP \cite{Gap}. 
Sample applications and a report on runtimes are included below.

In the appendix to this paper, two new infinite series of $L$-presented 
groups are described. The first series generalises the Fabrykowski-Gupta 
group \cite{FGu85}, the second series generalises the Gupta-Sidki group 
\cite{GSi83}.  

We investigated some of the groups in these series using the nilpotent 
quotient algorithm. As a result, we can conjecture that the generalised 
Fabrykowski-Gupta groups have finite width and we can identify a subfamily 
of these groups with width 2 and very slim lower central series. The 
generalised Gupta-Sidki groups do not exhibit a similarly obvious pattern 
in their lower central series factors and they could have infinite width.

\section{More about $L$-presentations}

In this section we recall the basic notions used to work with $L$-presented
groups and we exhibit some examples. First we note that an $L$-presentation 
$\langle S \mid Q \mid \Phi \mid R \rangle$ is called
\begin{items}
\bulit finite, if $S$,$Q$,$R$, and $\Phi$ are finite,
\bulit ascending, if $Q$ is empty, and
\bulit invariant, if 
$K = \langle Q \cup \bigcup_{\varphi \in \Phi^*} \varphi(R) \rangle^{F_S}$
satisfies $\varphi(K) \subseteq K$ for every $\varphi \in \Phi$.
\end{items}

In the remainder of this paper we are concerned with finite $L$-presentations
only. Invariant $L$-presentations will play an important role for our 
algorithm. We record some basic observations on $L$-presentations in the 
following remark.

\begin{remark}
\begin{items}
\bulit
Every ascending $L$-presentation is invariant. Conversely, if the 
$L$-presentation $\langle S \mid Q \mid \Phi \mid R \rangle$ is
invariant, then it determines the same group as the ascending 
$L$-presentation $\langle S \mid \emptyset \mid \Phi \mid Q \cup R 
\rangle$.
\bulit
Every finite presentation $\langle S \mid R \rangle$ can be written as
a finite $L$-presentation in the form $\langle S \mid R \mid \emptyset 
\mid \emptyset \rangle$ or in the form $\langle S \mid \emptyset \mid 
\{id\} \mid R \rangle$. The second form shows that every finite 
presentation can be written as a finite ascending $L$-presentation.
\end{items}
\end{remark}

Many of the well-known examples of $L$-presentations are invariant or even
ascending. A famous example for this case is the Grigorchuk group, see
\cite{Lys85} and \cite{Bar03} for details.

\begin{example} \label{examGr}
The Grigorchuk group can be defined by the following ascending (and hence
invariant) finite $L$-presentation.
  \[ \langle a,c,d \mid 
          \emptyset \mid 
          \{\sigma'\} \mid 
          \{a^2,[d,d^a],[d,d^{acaca}] \}
  \rangle, \]
with
  \[ \sigma'\colon\left\{\begin{array}{ccl}
    a&\mapsto&c^a\\
    c&\mapsto&cd\\
    d&\mapsto&c
  \end{array}\right\}. \]

We note that there are other finite $L$-presentations for the Grigorchuk
group known. An example is the following non-ascending, but invariant
$L$-presentation.
  \[
     \langle a,b,c,d \mid 
         \{a^2,b^2,c^2,d^2,bcd \} \mid
         \{ \sigma \} \mid 
         \{[d,d^a],[d,d^{acaca}]\}
     \rangle,
  \]
  with
  \[
  \sigma\colon\left\{\begin{array}{ccl}
    a&\mapsto&c^a\\
    b&\mapsto&d\\
    c&\mapsto&b\\
    d&\mapsto&c
  \end{array}\right\}.
  \]
\end{example}

\section{Polycyclic and nilpotent presentations}

Every finitely generated nilpotent group is polycyclic and hence can be 
described by a consistent polycyclic presentation. This type of presentation 
allows effective computations with the considered group and thus it
facilitates detailed investigations of the underlying group.

In this section we recall the definitions and some of the basic ideas on 
polycyclic presentations with particular emphasis on finitely generated
nilpotent groups. Further information and references can be found in 
\cite{HEO05}, Chapter X.

A {\em polycyclic presentation} is a presentation on a sequence of 
generators, $g_1, \ldots, g_n$ say, whose relations have the following
form for certain $r_1, \ldots, r_n \in \N \cup \{ \infty \}$:
\begin{eqnarray*}
g_i^{g_j} &=& g_{j+1}^{e_{i,j,j+1}} \cdots g_n^{e_{i,j,n}}
              \mbox{ for } j < i, \\
g_i^{g_j^{-1}} &=& g_{j+1}^{f_{i,j,j+1}} \cdots g_n^{f_{i,j,n}} 
              \mbox{ for } j < i, \mbox{ and }  \\
g_i^{r_i} &=& g_{i+1}^{l_{i,j,i+1}} \cdots g_n^{l_{i,j,n}} 
              \mbox{ for all } i \mbox{ with } r_i < \infty. \\
\end{eqnarray*}

Let $G$ be the group defined by the above presentation and let
$G_i = \langle g_i, \ldots, g_n \rangle \leq G$. Then the above relations
imply that the series $G = G_1 \unrhd G_2 \unrhd \ldots \unrhd G_n \unrhd 
G_{n+1} = \{1\}$ is a subnormal series with cyclic factors. We say that
this is the {\em polycyclic series} defined by the presentation. 

The factors of this polycyclic series satisfy $[G_i : G_{i+1}] \leq r_i$ 
for $1 \leq i \leq n$. The polycyclic presentation is called {\em consistent} 
if $[G_i : G_{i+1}] = r_i$ for $1 \leq i \leq n$. The consistency of a 
polycyclic presentation can be checked effectively, see \cite{Sim94}, page 424. 

Nilpotent presentations are a special case of polycyclic presentations for
finitely generated nilpotent groups. Let $G = \gamma_1(G) \geq \gamma_2(G) 
\geq \ldots$ denote the lower central series of $G$. Then we say that a 
polycyclic presentation of $G$ is a {\em nilpotent presentation} if its 
polycyclic series refines the lower central series of $G$.

A nilpotent presentation is called {\em weighted}, if there exists a 
function $w : \{g_1, \ldots, g_n\} \ra \N$ such that $w(g_k) = 1$ if 
and only if $g_k \not \in \gamma_2(G)$, and if $w(g_k) > 1$, then 
there exists a relation $g_i^{g_j} = g_i g_k$ with $j < i < k$ so that
$w(g_j) = 1$ and $w(g_i) = w(g_k)-1$.

\section{Computing abelian invariants}
\label{abelquot}

Let $G = \langle S \mid Q \mid \Phi \mid R \rangle$ be a group given by a
finite $L$-presentation. In this section we describe a method to determine 
the abelian invariants of $G$ and a corresponding consistent nilpotent 
presentation of the abelian group $G / \gamma_2(G) = G/G'$.

Our method is a direct generalisation of the well-known approach to determine
the abelian invariants of a finitely presented group. We refer to \cite{Sim94}
or \cite{HEO05} for further information.

Let $S = \{ s_1, \ldots, s_m \}$ and $F$ the free group on $S$.
Then every element $w \in F$ is a word in $S \cup S^{-1}$, say $w = 
s_{i_1}^{e_1} \cdots s_{i_l}^{e_l}$ with $e_i = \pm 1$. Define $a_j = 
\sum_{i_k = j} e_k \in \Z$ for $1 \leq j \leq m$ and let $\ol{w} = 
s_1^{a_1} \cdots s_m^{a_m}$. Then $\ol{w}$ can be considered as the 
collected word corresponding to $w$. It satisfies $wF' = \ol{w}F'$
and hence $\ol{w}$ is a representative of the coset $wF'$. Translating 
to additive notation, we can represent $\ol{w}$ by the vector $a_w = 
(a_1, \ldots, a_m) \in \Z^m$.

Every endomorphism $\varphi$ of $F$ satisfies $\varphi(F') \subseteq F'$ 
and hence induces an endomorphism $\ol{\varphi}$ of $F/F'$. Translating to 
additive notation as above, we can represent $\ol{\varphi}$ by a matrix 
$M_\varphi \in M_m(\Z)$ which acts by multiplication from the right on
$\Z^m$ as $\varphi$ acts on $F/F'$. Thus we obtain a homomorphism $End(F) 
\ra M_m(\Z) : \varphi \ms M_\varphi$. These constructions yield the 
following description of $G/\gamma_2(G)$.

\begin{lemma} \label{abelisom}
$G/ \gamma_2(G) \cong \Z^m / U_G$ where
$U_G = \langle a_q, a_r M_\varphi \mid 
             q \in Q, r \in R, \varphi \in \Phi^* \rangle$.
\end{lemma}

If a subgroup $V$ of $\Z^m$ is given by a finite set of generators, then 
algorithms for membership testing in $V$ and for computing the abelian 
invariants of the quotient $\Z^m/V$ together with a corresponding minimal 
generating set for this quotient are described in \cite{Sim94}, Chapter 8. 
Both methods rely mainly on Hermite normal form computations of matrices. The
latter allows to read off a consistent nilpotent presentation for $\Z^m/V$.

To apply these methods in our setting, it remains to determine a finite 
generating set for the subgroup $U_G$ of $\Z^m$ as defined in Lemma 
\ref{abelisom}. The following straightforward method achieves this aim.
Note that this method terminates, since ascending chains of subgroups
in $\Z^m$ terminate.

\begin{tabbing}
Fin\=ite\=Gen\=era\=tingSet( $U_G$ ) \\
\> initialise $U := \{ a_q, a_r \mid q \in Q, r \in R \}$ \\
\> initialise $T := \{ a_r \mid r \in R \}$ \\
\> while $T \neq \emptyset$ do \\
\> \> choose $t \in T$ and delete $t$ from $T$ \\
\> \> for $\varphi$ in $\Phi$ do \\
\> \> \> compute $s := t M_\varphi$ \\
\> \> \> if $s \not \in \langle U \rangle$ then 
             add $s$ to $U$ and add $s$ to $T$ \\
\> \> end for \\
\> end while \\
\> return $U$ 
\end{tabbing}

This completes our algorithm to determine the abelian invariants of $G$
and a consistent nilpotent presentation of $G/\gamma_2(G)$ for the group
$G$ which is given by a finite $L$-presentation. Further, this presentation
can be considered as weighted by assigning the weight 1 to every generator.

\section{Computing nilpotent quotients I}
\label{nilpquotI}

The algorithm of Section \ref{abelquot} generalises readily to a method 
for determining nilpotent quotients. This is straightforward to describe, 
but the resulting algorithm is usually not very effective in its applications. 
We include a description of this generalisation here for completeness and 
we refer to Section \ref{nilpquotII} for a significantly more effective 
approach towards computing nilpotent quotients.

Let $G = \langle S \mid Q \mid \Phi \mid R \rangle$ be a group given by a
finite $L$-presentation and let $n \in \N$. We wish to determine a consistent
polycyclic presentation for the quotient $G / \gamma_n(G)$. As above, let 
$F$ be the free group on $S$. Then a consistent polycyclic presentation for
a group $H$ with $H \cong F/\gamma_n(F)$ together with the corresponding 
natural epimorphism $\epsilon : F \ra H$ can be determined using a nilpotent 
quotient algorithm for finitely presented groups or theoretical background 
on free groups.

As $\gamma_n(F)$ is invariant under every endomorphism $\varphi$ of $F$, 
we obtain that $\varphi$ induces an endomorphism $\ol{\varphi}$ of the
quotient $F/\gamma_n(F)$. This endomorphism $\ol{\varphi}$ can be translated 
to an endomorphism $\ti{\varphi}$ of $H$ via $\epsilon$. Thus we obtain a
homomorphism $End(F) \ra End(H) : \varphi \ms \ti{\varphi}$. This setting
yields the following description of $G/\gamma_n(G)$.

\begin{lemma} \label{nilpisom}
$G / \gamma_n(G) \cong H / (U_G)^H$ where
$U_G = \langle \epsilon(q), \ti{\varphi} (\epsilon(r)) \mid 
             q \in Q, r \in R, \varphi \in \Phi^* \rangle$.
\end{lemma}

Let $V$ be a subgroup of $H$ given by a finite set of generators. Then
standard methods for polycyclically presented groups facilitate an effective
membership test in $V$, the computation of the normal closure of $V$ and 
the determination of a consistent polycyclic presentation of $H/V^H$. We 
refer to \cite{HEO05}, Chapter X, for background. 

Hence it remains to determine a finite generating set for the subgroup $U_G$
of $H$ as described in Lemma \ref{nilpisom} to complete our construction
for $G/\gamma_n(G)$. However, as ascending chains of subgroups in polycyclic
groups terminate, we can use the same method as in Section \ref{abelquot} 
to achieve this aim.

The main deficiency of this method is that it needs to compute a consistent
polycyclic presentation for the quotient $F/\gamma_n(F)$ and this quotient 
can easily be very large, even if the desired quotient $G/\gamma_n(G)$ is 
rather small.

\section{Computing nilpotent quotients II}
\label{nilpquotII}

Let $G$ be defined by a finite $L$-presentation and let $n \in \N$. In 
this section we describe a method to determine a consistent polycyclic 
presentation for $G / \gamma_n(G)$. First, in Section \ref{invar}, we 
consider the special case that $G$ is given by an invariant $L$-presentation. 
Then, in Section \ref{arbit}, we apply the special case method to obtain 
a method for the general case.

\subsection{Invariant finite $L$-presentations}
\label{invar}

Let $G = \langle S \mid Q \mid \Phi \mid R \rangle$ be a group given by a
finite invariant $L$-presentation and let $n \in \N$. We wish to determine 
a consistent nilpotent presentation for $G/\gamma_n(G)$. Note that the case 
$n=1$ is trivial and the case $n=2$ is covered by Section \ref{abelquot}. 
Hence we assume that $n \geq 3$ in the following.

Our overall idea generalises the method for finitely presented groups 
described by Nickel \cite{Nic95}. Thus our basic approach is an induction
on $n$. In the induction step, we assume that we have given a consistent 
weighted nilpotent presentation for $G/\gamma_{n-1}(G)$ and we seek to 
extend this to $G/\gamma_n(G)$. We discuss this step in more detail in 
the following.

First, we introduce some more notation. As before, let $F$ be the free 
group on $S = \{s_1, \ldots, s_m\}$ and let $K = \langle Q \cup 
\bigcup_{\varphi \in \Phi^*} \varphi(R) \rangle^F$ so that $G = F/K$. 
Define $K_n := K \gamma_n(F)$ for $n \in \N$. Then it follows that
\[ G/\gamma_n(G) \cong F / K_n \mbox{ for all } n \in \N. \]

As input for the induction step we use a {\em nilpotent quotient system}
for $F/K_{n-1}$ as described in \cite{Nic95}. We briefly recall the main
features of such a system as follows:
\begin{items}
\item[a)]
a consistent weighted nilpotent presentation $E/T$ defining a group $H$
and having the generators $e_1, \ldots, e_l$, say,
\item[b)]
a homomorphism $\tau : F \ra H$ with kernel $K_{n-1}$ which is defined by
the images $\tau(s_i) = w_i(e_1, \ldots, e_l)$ for $1 \leq i \leq m$, and
\item[c)]
for every $e_j$ with $w(e_j) = 1$ an index $i(j)$ such that the word 
$w_{i(j)}(e_1, \ldots, e_l)$ is of the form $w_{i(j)} = u_{i(j)} e_j$ where 
$u_{i(j)}$ is a word in $e_1, \ldots, e_{j-1}$.
\end{items}

The definition of a weighted nilpotent presentation incorporates that 
every generator of weight greater than 1 in $H$ can be written as a 
word in the generators of weight 1. Thus $H$ is generated by elements
of weight 1. Condition c) implies that for every generator of weight 1
we can compute a preimage in $F$. This yields that the homomorphism
$\tau$ is surjective and it follows that
\[ H \cong F/K_{n-1}.\]

The induction step now proceeds in two stages. First, we determine a 
nilpotent quotient system for $F / [K_{n-1},F]$ by extending the given 
nilpotent quotient system. An effective method for this purpose is 
described in \cite{Nic95}, Section 4. This yields
\begin{items}
\item[a)]
a consistent weighted nilpotent presentation $E^*/T^*$ defining a group
$H^*$ and having the generators $e_1, \ldots, e_l, e_{l+1}, \ldots, e_{l+d}$,
say,
\item[b)]
a homomorphism $\tau^* : F \ra H^*$ with kernel $[K_{n-1},F]$ which is
defined by images of the form $\tau^*(s_i) = w_i(e_1, \ldots, e_l) 
v_i(e_{l+1}, \ldots, e_{l+d})$ for $1 \leq i \leq m$, and
\item[c)]
for every $e_j$ with $w(e_j) = 1$ we have that $w_{i(j)}(e_1, \ldots, e_l) 
v_{i(j)}(e_{l+1}, \ldots, e_{l+d}) = u_{i(j)} e_j$ as above.
\end{items}

Note that $K_{n-1}/[K_{n-1},F]$ is a central subgroup in $F/[K_{n-1},F]$.
It corresponds via $\tau^*$ to the subgroup $M = \langle e_{l+1}, \ldots, 
e_{l+d} \rangle$ of the group $H^*$ so that $H^*$ is a central extension 
of $M$ by $H$.

As a second stage in the induction step of our algorithm, it now remains
to determine a nilpotent quotient system for $F/K_n$ from the given system
for $F/[K_{n-1},F]$. For this purpose we note that 
\[ K_n = K \gamma_n(F) =  K [K,F] [\gamma_{n-1}(F),F] = K [K_{n-1}, F]. \]
Thus it follows that 
\[ F/K_n \cong H^*/\tau^*(K)\]
and it remains to determine a finite generating set for $\tau^*(K)$ as
subgroup of the nilpotent group $H^*$. Once such a finite generating
set is given, we can then use standard methods for computing with 
polycyclically presented groups to determine a consistent weighted
nilpotent presentation for $H^*/\tau^*(K)$ and to modify the nilpotent 
quotient system for $F/[K_{n-1},F]$ to such a system for the quotient 
$F/K_n \cong H^*/\tau^*(K)$.

We investigate $\tau^*(K)$ in more detail in the following. Recall that
$M = \langle e_{l+1}, \ldots, e_{l+d} \rangle$ is an abelian subgroup of
$H^*$. 

\begin{lemma}
$\tau^*(K) \leq M$.
\end{lemma}

\begin{proof}
This follows directly, as $K \leq ker(\tau)$ and $\tau^*$ extends
$\tau$.
\end{proof}

Note that $M$ is a finitely generated abelian group by construction.
It now remains to determine a finite generating set for $\tau^*(K)$ as
a subgroup of $M$.

\begin{lemma}
Every endomorphism $\varphi \in \Phi^*$ induces an endomorphism 
$\ol{\varphi} \in End(M)$ via $\tau^*$ and we obtain a homomorphism
$\Phi^* \ra End(M) : \varphi \ms \ol{\varphi}$.
\end{lemma}

\begin{proof}
Let $\varphi \in \Phi^*$. As the given $L$-presentation is invariant, it 
follows that $\varphi(K) \subseteq K$ holds. Clearly also $\gamma_i(F)$ is 
invariant under $\varphi$ for every $i \in \N$. Thus we obtain that 
$K_{n-1} = K \gamma_{n-1}(F)$ and also $[K_{n-1},F]$ are invariant under 
$\varphi$. Thus $\varphi$ induces an endomorphism of $K_{n-1}/[K_{n-1},F]$ 
and hence, via $\tau^*$, also of $M$.
\end{proof}

This implies the following.

\begin{lemma}
\label{taugens}
$\tau^*(K) = \langle \tau^*(q), \ol{\varphi}(\tau^*(r)) \mid q \in Q, 
r \in R, \varphi \in \Phi^* \rangle$.
\end{lemma}

\begin{proof}
This follows directly by translating the defining generating set of $K$
to generators of $\tau^*(K) \leq M$.
\end{proof}

As $M$ is finitely generated abelian, it satisfies the ascending chain
condition. Thus a finite generating set for $\tau^*(K)$ can be computed 
from the description given in Lemma \ref{taugens} using a similar approach
to the algorithm `FiniteGeneratingSet' of Section \ref{abelquot}. 

We summarise our resulting algorithm for the induction step as follows. Let 
${\mathcal Q}(F/L)$ denote the nilpotent quotient system for a quotient 
$F/L$ of $F$.

\begin{figure}[htb]
\begin{tabbing}
Ind\=uct\=ion\=Ste\=p( ${\mathcal Q}(F/K_{n-1})$ ) \\
\> Compute a nilpotent quotient system ${\mathcal Q}(F/[K_{n-1},F])$ 
   (see \cite{Nic95}). \\
\> Induce every $\varphi \in \Phi$ to $\ol{\varphi} \in End(M)$. \\
\> Induce every $g \in Q \cup R$ to $\tau^*(g) \in M$. \\
\> Determine a finite generating set for $\tau^*(K)$ using 
   Lemma \ref{taugens} and `FiniteGeneratingSet'. \\
\> Determine a consistent weighted nilpotent presentation for 
   $H^*/\tau^*(K)$. \\
\> Return ${\mathcal Q}(F/K_n)$ as modification of 
   ${\mathcal Q}(F/[K_{n-1},F])$.
\end{tabbing}
\end{figure}

\subsection{Arbitrary finite $L$-presentations}
\label{arbit}

Now let $G = \langle S \mid Q \mid \Phi \mid R \rangle$ be a group given by 
an arbitrary finite $L$-presentation and let $n \in \N$. We wish to determine 
a consistent polycyclic presentation for $G/\gamma_n(G)$. As above, let $F$ 
be the free group on $S$ and denote $K = \langle Q \cup \bigcup_{\varphi \in 
\Phi^*} \varphi(R) \rangle^F$. Our method proceeds in the following 3 steps.
\medskip

Step 1: We determine an invariant finite $L$-presentation $\langle S \mid 
\ol{Q} \mid \Phi \mid R \rangle$ defining a group $\ol{G}$, say, so that 
its kernel $\ol{K} = \langle \ol{Q} \cup \bigcup_{\varphi \in \Phi^*} 
\varphi(R) \rangle^F$ satisfies $\ol{K} \subseteq K$.
\medskip

Step 2: We determine the nilpotent quotient of the larger group $\ol{G}$ as
$H := \ol{G}/\gamma_n(\ol{G})$ using the method of Section \ref{invar}. 
\medskip

Step 3: We determine the finite set $U$ of images of $Q \setminus \ol{Q}$ in  
$H$ and obtain $G / \gamma_n(G) \cong H / \langle U \rangle^H$ using standard
methods for polycyclically presented groups.
\medskip

Step 1 requires some further explanation. First note that we could always
choose $\ol{Q} = \emptyset$ and thus obtain a fully automatic algorithm.
However, the effectivity of the above method relies critically on finding
an $L$-presentation in Step 1 that yields a possibly ``small'' subgroup 
$\langle U \rangle^H$. (``Small'' means here that the difference in the
numbers of generators of the polycyclic presentation for $H$ and its 
induced presentation for $H / \langle U \rangle^H$ is small.) Thus it may
be of interest to supply a ``nice'' $L$-presentation for Step 1 by other 
means. However, there is no general algorithm for finding such a ``nice''
$L$-presentation available at current.

\section{Sample applications and runtimes}

The algorithm described in this paper has been implemented in the \GAP 
package {\sf NQL} \cite{NQL}. In this section we outline runtimes for some 
sample applications of this algorithm and thus exhibit the scope and 
the range of possible applications of our algorithm. 

All timings displayed below have been obtained on an Intel Pentium 4 
computer with clock speed 2.80 GHz by applying the {\sf NQL} algorithm with 
a time limit of two hours. Then the computation has been stopped and 
the resulting nilpotent quotient together with the total time used to 
obtain this quotient has been listed.

\subsection{Some well-known groups}

There are various interesting examples of finitely $L$-presented, but not 
finitely presented groups known. We list some of them in the following;
the Fabrykowski-Gupta group and the Gupta-Sidki group are treated in detail 
in the next sections.

\begin{items}
\bulit 
$G$: the Grigorchuk group with its $L$-presentation in \cite{Lys85}.
\bulit
$\ti{G}$: the Grigorchuk supergroup with its $L$-presentation in 
\cite{Bar03}, Theorem 4.6.
\bulit
$BSV$: the Brunner-Sidki-Vieira group \cite{BSV99} with its 
$L$-presentation in \cite{Bar03}, Theorem 4.4. 
\bulit
$\Delta$: the Basilica group \cite{GZu02} with its $L$-presentation 
in \cite{BVi05}.
\bulit
$B$: the Baumslag group \cite{Bau71b} with its $L$-presentation in 
\cite{Bar03}, Theorem 4.2. 
\bulit
$L$: the Lamplighter group with its $L$-presentation in \cite{Bar03}, 
Theorem 4.1. 
\end{items}

Table 1 briefly describes the application of our algorithm to these groups. 
It lists whether the considered groups have ascending or non-invariant 
$L$-presentations, it briefly describes the obtained nilpotent quotients 
by their classes and the number of generators in their nilpotent 
presentations and it exhibits the runtimes used to determine the nilpotent
quotients. 

Table 1 shows that our algorithm has a significantly better performance on 
ascending $L$-presentations than on non-invariant ones. In the case of a 
non-invariant $L$-presentation, the column `gens' of Table 1 lists in 
brackets the number of generators of the invariant $L$-presentation used 
in Step 1 of the method in Section \ref{arbit}.

\begin{table}[htb]
\begin{center}
\label{table1}
\begin{tabular}{|c|c|c|c|c|}
\hline
Group      & prop  & class & gens & time (h:min) \\
\hline
$G$        & asc   & 80    & 130  & 1:53  \\
$\ti{G}$   & asc   & 47    & 127  & 1:56  \\
$BSV$      & asc   & 34    & 171  & 1:27  \\
$\Delta$   & asc   & 39    & 220  & 1:47  \\
\hline
$B$        & non-inv & 11  & 12 (423) & 0:21  \\
$L$        & non-inv & 9   & 10 (253) & 0:04  \\
\hline
\end{tabular}
\caption{Some well-known groups}
\end{center}
\end{table}

In the remainder of this subsection, we outline and discuss the lower 
central series quotients $\gamma_i(*) / \gamma_{i+1}(*)$ for the groups in 
Table 1 in more detail. To shorten notation, we outline lists in
collected form; that is, if an entry $a$ in a list appears in $n$ consecutive 
places, then we write $a^{[n]}$ instead of $n$ times $a$. 

The lower central series quotients of the Grigorchuk group $G$ are known 
by theoretical results of Rozhkov \cite{Roz96}, see also \cite{Gri99}. Our
computations confirm the following theorem.

\begin{theorem} {\rm (See \cite{Roz96})}
The Grigorchuk group $G$ satisfies that
\[ rk(\gamma_i(G)/\gamma_{i+1}(G)) = 
   \left\{ \begin{array}{cl} 
  3 \mbox{ or } 2 & \mbox{ if } i=1 \mbox{ or } 2 \mbox{ resp. } \\
  2 & \mbox{ if } i \in \{2\cdot2^k+1, \ldots, 3\cdot 2^k \} \\
  1 & \mbox{ if } i \in \{3\cdot2^k+1, \ldots, 4\cdot 2^k \} \\
   \end{array} \right\} \mbox{ with } k \in \N_0.\]
\end{theorem}

For the Grigorchuk supergroup $\ti{G}$ we computed 
$\gamma_i(\ti{G})/ \gamma_{i+1}(\ti{G})$ for $1 \leq i \leq 64$. 
The resulting groups are elementary abelian 2-groups with ranks
\[ 4, 3^{[2]}, 2, 3^{[2]}, 2^{[2]}, 3^{[4]}, 2^{[4]}, 3^{[8]}, 2^{[8]}, 
   3^{[16]}, 2^{[16]}. \]
This induces the following conjecture.

\begin{conjecture} The Grigorchuk supergroup $\ti{G}$ satisfies that
\[ rk(\gamma_i(\ti{G})/\gamma_{i+1}(\ti{G})) = 
   \left\{ \begin{array}{cl} 
   3 & \mbox{ if } i \in \{ 2 \cdot 2^{k}+1, \ldots, 3 \cdot 2^k \} \\
   2 & \mbox{ if } i \in \{ 3 \cdot 2^{k}+1, \ldots, 4 \cdot 2^k \} \\
   \end{array} \right\} \mbox{ with } k \in \N_0.\]
\end{conjecture}

For the Brunner-Sidki-Vieira group $BSV$ the Jennings series is completely 
determined in \cite{Bar04}. But there have been only the first 4 quotients 
of its lower central series known so far. We computed $\gamma_i(BSV)/ 
\gamma_{i+1}(BSV)$ for $1 \leq i \leq 43$ and obtained the following 
abelian invariants:
\begin{eqnarray*}
&\mbox{}&    (  0,  0 ), ( 0 ), (  8 ), \\
&\mbox{}&    (  8 ), 
             (  4,  8 ), 
             (  2,  8 ), \\
&\mbox{}&    (  2,  2,  8 )^{[2]}, 
             (  2,  2,  4,  8 )^{[2]}, 
             (  2,  2,  2,  8 )^{[2]}, \\
&\mbox{}&    (  2,  2,  2,  2,  8 )^{[4]}, 
             (  2,  2,  2,  2,  4,  8 )^{[4]}, 
             (  2,  2,  2,  2,  2,  8 )^{[4]}, \\
&\mbox{}&    (  2,  2,  2,  2,  2,  2,  8 )^{[8]},
             (  2,  2,  2,  2,  2,  2,  4,  8 )^{[8]},
             (  2,  2,  2,  2,  2,  2,  2,  8 )^{[3]} 
\end{eqnarray*}
This induces the following conjecture, where $I(*)$ denotes the abelian 
invariants of a group.

\begin{conjecture} The Brunner-Sidki-Vieira group $BSV$ satisfies that
\[ I( \gamma_i(BSV)/\gamma_{i+1}(BSV) ) = 
   \left\{ \begin{array}{ll} 
   (2^{[2k]},8)    & \mbox{ if } i \in \{ 3 \cdot 2^k+1, \ldots, 4 \cdot 2^k \}\\
   (2^{[2k]},4,8)  & \mbox{ if } i \in \{ 4 \cdot 2^k+1, \ldots, 5 \cdot 2^k \}\\
   (2^{[2k+1]},8)  & \mbox{ if } i \in \{ 5 \cdot 2^k+1, \ldots, 6 \cdot 2^k \}\\
   \end{array} \right\} \mbox{ with } k \in \N_0. \]
\end{conjecture}

For the Basilica group $\Delta$ we computed 
$\gamma_i(\Delta)/ \gamma_{i+1}(\Delta)$ for $1 \leq i \leq 48$ 
and obtained the following abelian invariants:
\begin{eqnarray*}
&\mbox{}&    (  0,  0 ), (  0 ), (  4 )^{[2]}, (  4,  4 ), (2,4) \\
&\mbox{}&    (  2,  2,  4 )^{[2]}, 
             (  2,  2,  2,  4 ), 
             (  2,  2,  2,  2,  4 )^{[2]},
             (  2,  2,  2,  4 ), \\
&\mbox{}&    (  2,  2,  2,  2,  4 )^{[4]},
             (  2,  2,  2,  2,  2,  4 )^{[2]},
             (  2,  2,  2,  2,  2,  2,  4 )^{[4]},
             (  2,  2,  2,  2,  2,  4 )^{[2]}, \\
&\mbox{}&    (  2,  2,  2,  2,  2,  2,  4 )^{[8]},
             (  2,  2,  2,  2,  2,  2,  2,  4 )^{[4]},
             (  2,  2,  2,  2,  2,  2,  2,  2,  4 )^{[8]},
             (  2,  2,  2,  2,  2,  2,  2,  4 )^{[4]}, \\
\end{eqnarray*}
This induces the following conjecture.

\begin{conjecture} The Basilica group $\Delta$ satisfies that
\[ I(\gamma_i(\Delta)/\gamma_{i+1}(\Delta)) = 
   \left\{ \begin{array}{cl} 
   (2^{[2k+2]},4) & \mbox{ if } i \in \{ 6 \cdot 2^k+1, \ldots, 8 \cdot 2^k \}\\
   (2^{[2k+3]},4) & \mbox{ if } i \in \{ 8 \cdot 2^k+1, \ldots, 9 \cdot 2^k \}\\
   (2^{[2k+4]},4) & \mbox{ if } i \in \{ 9 \cdot 2^k+1, \ldots, 11 \cdot 2^k \}\\
   (2^{[2k+3]},4) & \mbox{ if } i \in \{ 11 \cdot 2^k+1, \ldots, 12 \cdot 2^k \}\\
   \end{array} \right\} \mbox{ with } k \in \N_0. \]
\end{conjecture}

Baumslag's group $B$ and the Lamplighter group $L$ are both known to be
metabelian. This yields that their lower central series patterns can be
deduced theoretically. We include the abelian invariants of $\gamma_i(*)
/ \gamma_{i+1}(*)$ for these two groups as far as we computed them for
completeness:
\[ \mbox{ for $B$: } (3, 0), 3^{[10]} \;\;\;\;\;\;\;\;
   \mbox{ for $L$: } (2, 0), 2^{[8]}. \]

\subsection{The Fabrykowski-Gupta group and its generalisation}

An infinite series of groups with an ascending finite $L$-presentation are 
introduced in Appendix \ref{appB}: the {\em generalised Fabrykowski-Gupta 
groups} $\Gamma_p$ with $p \geq 3$. We used these groups as sample inputs 
for the nilpotent quotient algorithm; the results are outlined in this 
section. They support that this family of groups contains some very 
interesting groups.

First, we briefly summarise the results of our algorithm on $\Gamma_p$ for
some small $p$ in Table 2 using the same format as in Table 1. Note that
all considered $L$-presentations are ascending in this case. Additionally,
the table contains a column noting whether our algorithm found a maximal 
nilpotent quotient.

\begin{table}[htb]
\begin{center}
\label{table2}
\begin{tabular}{|c|c|c|c|c|}
\hline
Group      & max quot & class & gens & time (h:min) \\
\hline
$\Gamma_3$    & no  & 71 & 112 & 1:50  \\
$\Gamma_4$    & no  & 66 & 146 & 1:55  \\
$\Gamma_5$    & no  & 53 &  60 & 1:58  \\
$\Gamma_6$    & yes &  3 &   4 & 0:00  \\
$\Gamma_7$    & no  & 44 &  50 & 1:37  \\
$\Gamma_8$    & no  & 52 & 116 & 1:47  \\
$\Gamma_9$    & no  & 58 &  84 & 1:54  \\
$\Gamma_{10}$ & yes &  5 &   6 & 0:00  \\
$\Gamma_{11}$ & no  & 33 &  35 & 1:48  \\
$\Gamma_{12}$ & yes &  6 &   7 & 0:00  \\
$\Gamma_{14}$ & yes &  7 &   8 & 0:00  \\
$\Gamma_{15}$ & yes &  5 &   6 & 0:00  \\
$\Gamma_{18}$ & yes & 15 &  16 & 0:06  \\
$\Gamma_{20}$ & yes &  6 &   7 & 0:02  \\
$\Gamma_{21}$ & yes &  7 &   8 & 0:04  \\
\hline
\end{tabular}
\caption{Fabrykowski-Gupta groups $\Gamma_p$ for some small $p$}
\end{center}
\end{table}

In the following we discuss the lower central series factors of the
groups $\Gamma_p$ in more detail. First, we consider the case that
$p$ is not a prime-power. We summarise our results in the following 
conjecture.

\begin{conjecture}
If $p$ is not a prime-power, then $\Gamma_p$ has a maximal nilpotent
quotient.
\end{conjecture}

Next, we consider the case that $p$ is a prime. For the smallest possible
prime $p=3$, there is a theoretical description of the lower central series
factors of $\Gamma_3$ known from \cite{Bar05}. Our computations confirm
the following theorem.

\begin{theorem} {\rm (See \cite{Bar05})}
\[ rk(\gamma_i(\Gamma_3)/\Gamma_{i+1}(\Gamma_3)) = 
   \left\{ \begin{array}{cl} 
  2 \mbox{ or } 1 & \mbox{ if } i= 1 \mbox{ or } 2 \mbox{ resp. } \\
  2 & \mbox{ if } i \in \{3^k+2,\ldots, 2\cdot 3^k+1\} \\
  1 & \mbox{ if } i \in \{2\cdot 3^k+2,\ldots, 3^{k+1}+1 \} \\
   \end{array} \right\} \mbox{ with } k \in \N_0. \]
\end{theorem}

For the primes $p = 5,7,11$, we list the lower central series factors
$\gamma_i(\Gamma_p) / \gamma_{i+1}(\Gamma_p)$ obtained by our algorithm
in the following. Note that all determined factors are elementary abelian 
$p$-groups and we list their ranks in collected form. 

\begin{items}
\bulit $\Gamma_5$:
$\;\;\;\; 2, 1^{[3]}, 2, 1^{[13]}, 2^{[5]}, 1^{[30]}$.
\bulit $\Gamma_7$:
$\;\;\;\; 2, 1^{[5]}, 2, 1^{[33]}, 2^{[4]}$.
\bulit $\Gamma_{11}$:
$\;\;\; 2, 1^{[9]}, 2, 1^{[22]}$.
\end{items}

Thus if $p$ is a prime, then the groups $\Gamma_p$ seem to have a very slim 
lower central series. It seems very likely that these groups exhibit a lower 
central series pattern similar to that of $\Gamma_3$ and it would be very 
interesting to spot and prove this. However, for this purpose a larger 
computed sequence would be helpful. We only formulate the following conjecture.

\begin{conjecture}
Let $p$ be an odd prime. Then $\Gamma_p$ is a group of width 2.
\end{conjecture}

Finally, we consider the case that $p$ is a prime-power, say $q^n$. All the
obtained lower central series factors $\gamma_i(\Gamma_p) / 
\gamma_{i+1}(\Gamma_p)$ are $q$-groups in this case and, except for some 
initial entries, they are elementary abelian. Again, it would be interesting 
to find and prove a general pattern for these factors.

\begin{items}
\bulit $\Gamma_4$: 
$\;\;\;\; (4,4), (4), 2^{[4]}, 3^{[3]}, 2^{[13]}, 3^{[12]}, 2^{[32]}$. 
\bulit $\Gamma_8$:
$\;\;\;\; (8,8),(8),(4)^{[4]}, 2, 1, 2^{[2]}, 3, 2, 3^{[2]}, 4, 3^{[8]}, 
          2^{[23]}, 3^{[5]}, 2$.
\bulit $\Gamma_9$:
$\;\;\;\; (9,9),(9)^{[2]}, 1^{[5]}, 2^{[6]}, 3, 2^{[17]}, 1^{[26]}$.
\end{items}

Thus the groups $\Gamma_p$ still seem to be of finite width, but the width
grows with the exponent $n$ in the power $p = q^n$. 

\subsection{The Gupta-Sidki group and its generalisations}

The Gupta-Sidki group $GS$ has originally been introduced in \cite{GSi83}
and has become famous for its role in connection with the Burnside problems.
In Appendix \ref{appC} generalisations $GS_p$ of this group for all odd
primes $p$ are introduced and finite non-invariant $L$-presentation for
these groups are obtained. In this section, we investigate the groups 
$GS_p$ using our nilpotent quotient algorithm.

As a preliminary step, we discuss two different strategies to determine 
nilpotent quotients of $GS_p$ with our algorithm. First, we can apply our 
algorithm to the non-invariant $L$-presentation of $GS_p$ outlined in 
Appendix \ref{appC}. This is straightforward, but usually yields only
very limited results, as our algorithm is not effective on non-invariant 
$L$-presentations.

For a second, more effective approach we use the structure of $GS_p$ as
exhibited in Appendix \ref{appC}. Every $GS_p$ is of the form $GS_p \cong
D_p \rtimes C_p$, where $D_p$ is generated by $\{\sigma_1, \ldots, \sigma_p\}$ 
and the cyclic group $C_p$ acts by permuting these elements cyclically. An 
ascending $L$-presentation for $D_p$ is also included in Appendix \ref{appC}.
Now we can apply our algorithm to the ascending $L$-presentation of $D_p$ 
and determine $D_p/ \gamma_c(D_p)$ for some $c$. Then, defining $H_p = 
D_p/\gamma_c(D_p) \rtimes C_p$, we obtain $GS_p/ \gamma_i(GS_p) 
\cong H_p / \gamma_i(H_p)$ for all $i \leq c$. 

Table 3 summarises runtimes and a brief overview on the results of our 
algorithm applied to $GS_p$ for $p = 3,5,7$. The table uses the same 
notation as Table 1. Instead of a column `prop' it has a column `strategy'
which lists the used strategy and hence also determines whether our 
algorithm was applied to an ascending or non-invariant $L$-presentation.
Note that we applied the nilpotent quotient algorithm for 2 hours in all 
cases. Thus the runtimes for $GS_3$ with strategy 1 show that the first 5 
quotients are fast to obtain, while the 6th quotients takes over 2 hours 
and hence did not complete. Further, Table 3 shows that strategy 2 is more 
successful on $GS_3$ than strategy 1; a feature that we also observed for 
other $GS_p$. 

\begin{table}[htbp]
\begin{center}
\begin{tabular}{|c|c|c|c|c|}
\hline
Group    & strategy & class & gens & time (h:min) \\
\hline
$GS_3$   & 1 & 5     & 8 (215)  & 0:02  \\
\hline
$GS_3$   & 2 & 25    & 51  & 1:44  \\
$GS_5$   & 2 & 9     & 22  & 1:09  \\
$GS_7$   & 2 & 6     & 13  & 0:59  \\
\hline
\end{tabular}
\caption{The Gupta-Sidki groups $GS_p$ for some small primes $p$}
\end{center}
\end{table}

Next, we discuss the obtained results for the lower central series of
$\Gamma_p$ and $H_p$ in more detail. Our computational results for $GS_3$ 
agree with the following theoretical description of $\gamma_i(GS_3)/
\gamma_{i+1}(GS_3)$ from \cite{Bar05}.

\begin{theorem} {\rm (See \cite{Bar05})} \\
Let $\alpha_1=1$, $\alpha_2=2$, and $\alpha_n=2\alpha_{n-1}+\alpha_{n-2}$ 
for $n\ge 3$.  Then, for $n\ge2$, the rank of 
$\gamma_n(GS_3)/\gamma_{n+1}(GS_3)$ is the number of ways of writing $n-1$ 
as a sum $k_1\alpha_1+\dots+k_t\alpha_t$ with all $k_i\in\{0,1,2\}$.
\end{theorem}

For all primes $p>3$ no theoretical description of the lower central 
series factors of $GS_p$ is available; it would be very interesting to 
obtain one. In the following we outline our computed results 
for the ranks of $\gamma_i(H_p)/\gamma_{i+1}(H_p)$. These are isomorphic 
to $\gamma_i(GS_p)/\gamma_{i+1}(GS_p)$ for all $i \leq c$, where $c$ is 
the class listed in Table 3, and they are epimorphic images otherwise. 
This is indicated by a bar $\mid$ in the list below. 

\begin{items}
\bulit $H_5$:
$2, 1, 2^{[2]}, 3, 2, 3^{[2]}, 4 \; \mid \; 
4^{[3]}, 3^{[3]}, 4^{[4]}, 3, 4^{[2]}, 6^{[3]}, 5, 4, 2^{[3]}, 1^{[3]}$.
\bulit $H_7$: 
$2, 1, 2^{[2]}, 3^{[2]}, 4 \; \mid \; 3, 4^{[2]}, 5^{[6]}, 4, 3^{[5]}, 
2^{[3]}, 1^{[2]}$.
\end{items}

\subsection{Some finitely presented groups}

Each finitely presented group $\langle X\mid R\rangle$ has a finite 
ascending $L$-presentation of the form 
$\langle X\mid\emptyset\mid\{{\rm id}\}\mid R\rangle$ and 
hence the algorithm described here also applies to finitely presented 
groups. The following finitely presented groups are from \cite{Nic95}.
\begin{eqnarray*}
  G_1&=&\textrm{free group on 3 generators}\\
  G_2&=&\textrm{free group on 4 generators}\\
  G_3&=&\langle a,b\mid [a,[a,[a,b]]],[b,[b,[a,b]]]\rangle\\
  G_4&=&\langle x,y\mid [[y,x],y],[[[[[y,x],x],x],x],x] \rangle.
\end{eqnarray*}

Runtimes for these groups are outlined Table 4.

\begin{table}[htpb]
\begin{center}
\begin{tabular}{|c|c|c|c|}
\hline
Group       & class & gens  & time \\
\hline
$G_1$       & 8     & 1318  & 0:11 \\
$G_2$       & 6     &  964  & 0:04 \\
$G_3$       & 17    &  272  & 1:19 \\
$G_4$       & 20    &  275  & 1:31 \\
\hline
\end{tabular}
\caption{Some finitely presented groups}
\end{center}
\end{table}

Comparing these runtimes with the runtimes of the nilpotent quotient
algorithm of the {\sf NQ} package \cite{NQ} shows that the latter is 
significantly faster. This is mainly due to the fact that the {\sf NQL} 
package \cite{NQL} is implemented in \GAP code and uses the available
machinery for computing with polycyclic groups in \GAP, while the 
{\sf NQ} package \cite{NQ} is implemented in C code and all underlying
machinery has been designed for computing with nilpotent groups. 

\section{Appendix (by Laurent Bartholdi)}

One of the main reasons to introduce $L$-presentations was the desire to
understand better some examples of ``self-similar'' groups, and in
particular striking patterns along their lower central series.

By a self-similar group we mean a group $G$ acting on the set of words
$X^*$ over an alphabet $X$, and preserving the length and `prefix'
relation on $X^*$. This means that for every $x\in X,g\in G$ there are
$y\in X,h\in G$ with

\begin{equation}\label{eq:ss}
  g(xw)=yh(w) \mbox{ for all }w\in X^*.
\end{equation}

These groups have appeared across a wide range of mathematics, answering 
classical questions, for example on torsion and growth, in infinite group 
theory as well as establishing new links with complex dynamics. See the 
monograph~\cite{Nek05} for more details. 

We will capture the data in equation (\ref{eq:ss}) defining a self-similar 
group as follows: $(y,h)=\Psi(g,x)$, for some function $\Psi:G\times X
\to X\times G$. It suffices to specify $\Psi$ on $S\times X$ for some 
generating set $S$ of $G$.

In \cite{Bar03,BSV99} a few `sporadic' calculations of 
$L$-presentations for self-similar groups are described. A systematic 
construction of $L$-presentations for all self-similar groups arising as 
the iterated monodromy group of a quadratic, complex polynomial is given
in \cite{BMa06}. 

In this appendix, I describe two more infinite families of $L$-presentations.
These naturally generalize the groups constructed by Gupta and Sidki 
\cite{GSi83} and Fabrykowski and Gupta \cite{FGu85}. The former is an 
elementary family of infinite, finitely generated, torsion $p$-groups, 
while the latter is a group of intermediate word-growth. The `generalized' 
groups that I consider here are a small variations of their constructions.

I will present the calculations in compact form, mainly relying on Theorem
3.1 in \cite{Bar03}. They generalize the case $p=3$ described in that paper, 
correcting at the same time typographical and/or copying errors. The 
presentations given in \cite{Bar03} for the Gupta-Sidki and Fabrykowski-Gupta
groups are not correct as is. The presentations given in \cite{Sid87} and 
\cite{Gup84} are in principle equivalent, but not as readily amenable to 
manipulation and understanding.

The obtained $L$-presentations may be accessed in the forthcoming 
\textsf{GAP} package \textsf{FR}, see \cite{fr}.

\subsection{Generalized Fabrykowski-Gupta groups}
\label{appB}

For a fixed integer $p\ge3$ let $X=\Z/p\Z=\{0,\dots,p-1\}$ and consider 
the group $G=\langle a,r\rangle$ acting on $X^*$ via
\[\Psi(a,x)=(x+1,1),\qquad\Psi(r,0)=(0,r),\quad\Psi(r,1)=(1,a),
   \quad\Psi(r,x)=(x,1)\mbox{ else}.\]
In the case $p=3$, this is the Fabrykowski-Gupta group, which was shown in
\cite{FGu85} to be of ``subexponential word-growth''. The following 
theorem states the main result of this section.

\begin{theorem} \label{gFG}
For any $p$, the generalized Fabrykowski-Gupta group $G$ admits a finite 
ascending $L$-presentation with generators $\alpha,\rho$. With $\sigma_i 
= \rho^{\alpha^i}$ for $1 \leq i \leq p$, its iterated relations are 
\[\RELS=\left\{
    \alpha^p, 
    \left[\sigma_i^{\sigma_{i-1}^n},\sigma_j^{\sigma_{j-1}^m}\right],
    \sigma_i^{-\sigma_{i-1}^{n+1}} 
    \sigma_i^{\sigma_{i-1}^n\sigma_{i-1}^{\sigma_{i-2}^m}}
    \bigg|
    \begin{array}{c}
          1\le i,j\le p \\ 2\le|i-j|\le p-2 \\ 0\le m,n\le p-1
    \end{array}
    \right\}, \]
and its only endomorphism is defined by $\varphi(\alpha) = \rho^{\alpha^{-1}}$ 
and $\varphi(\rho) = \rho$. (Note that some relators in $\RELS$ are redundant,
since the elements $\sigma_1, \ldots, \sigma_p$ are conjugate. For example,
one may fix $i=1$.)
\end{theorem}

The proof of this theorem follows the strategy of~\cite{Bar03}, Theorem
3.1, which proceeds as follows. We first consider the finitely presented 
group
\[\Gamma=\langle\alpha,\rho|\alpha^p,\rho^p\rangle\]
mapping naturally onto $G$ by `greek$\mapsto$latin'. We then consider
the subgroup $\Delta=\langle\rho^{\alpha^i}:0\le i<p\rangle$ of $\Gamma$, 
and the homomorphism $\Phi:\Delta\to\Gamma^p$, defined by 
\[\Phi(\rho^{\alpha^i})=(1,\dots,\rho,\alpha,\dots,1)
  \mbox{ with the $\rho$ at position $i$.} \]
We compute a presentation of $\Phi(\Delta)$; the kernel of $\Phi$ is
generated by the set $\RELS$ of $\Phi$-preimages of relators in that
presentation. Finally, we seek a section $\Sigma:\Gamma\to\Delta$ of 
the projection of $\Phi$ on its first coordinate. We then have, for all 
$x\in\Gamma$, 
\[\Phi(\Sigma(x))=(x,?,\dots,?),\] 
where the $?$ stand for unimportant elements of $\Gamma$. The following
result now allows to read off a finite ascending $L$-presentation for $G$.

\begin{scholium}
An $L$-presentation of $G$ is given by generators $\alpha,\rho$;
endomorphism $\Sigma$; and iterated relations $\RELS$.
\end{scholium}

In the remainder of this section, we apply this strategy to determine a
finite $L$-presentation for $G$ and thus prove Theorem \ref{gFG}.
A presentation of $\Phi(\Delta)$ can be determined by the 
Reidemeister-Schreier method. Consider first the presentation
\[\Pi=\langle\alpha_1,\dots,\alpha_p,\rho_1,\dots,\rho_p|\alpha_i^p,
      \rho_i^p,[\alpha_i,\alpha_j],[\alpha_i,\rho_j],[\rho_i,\rho_j]
      \mbox{ for }i\neq j\rangle;\]
this is a presentation of $\Gamma^p$, and $\Phi(\Delta)$ is the
subgroup $\langle \sigma_i:=\rho_i\alpha_{i+1}\rangle$. Here and below
indices are all treated modulo $p$. We rewrite this presentation as
\[\Pi=\langle\alpha_1,\dots,\alpha_p,\sigma_1,\dots,\sigma_p|\alpha_i^p,
      \sigma_i^p,[\alpha_i,\alpha_j],[\alpha_i,\sigma_j],
      [\sigma_i\alpha_{i+1}^{-1},\sigma_j\alpha_{j+1}^{-1}]
      \mbox{ for }i\neq j\rangle.\]
Next we rewrite the last set of relations either as 
$[\sigma_i,\sigma_j]$, if $2\le|i-j|\le p-2$, or as
$\sigma_i^{\alpha_i}=\sigma_i^{\sigma_{i-1}}$, in the other cases.

We choose as Schreier transversal all $p^p$ elements $\alpha_1^{n_1}\dots
\alpha_p^{n_p}$. The Schreier generating set easily reduces to 
$\{\sigma_{i,n}:=\sigma_i^{\alpha_i^{n_i}}\}$. The Schreier relations are 
all $[\sigma_{i,n},\sigma_{j,m}]$ for $2\le|i-j|\le p-2$, and all 
$\sigma_{i,n+1}=\sigma_{i,n}^{\sigma_{i-1,m}}$.

In particular, we can use this last relation (with $m=0$) to eliminate
all generators $\sigma_{i,n}$ with $n\neq0$, replacing them by
$\sigma_i^{\sigma_{i-1}^n}$. We obtain $\Phi(\Delta)=\langle\sigma_1,
\dots,\sigma_p|\sigma_1^p,\dots,\sigma_p^p, \RELS \rangle$, with
\[\RELS=\left\{\left[\sigma_i^{\sigma_{i-1}^n},
              \sigma_j^{\sigma_{j-1}^m}\right]
              \mbox{ whenever }2\le|i-j|\le p-2,\quad
              \sigma_i^{-\sigma_{i-1}^{n+1}}\sigma_i^{\sigma_{i-1}^n
              \sigma_{i-1}^{\sigma_{i-2}^m}}\right\}.\]
Note that $\Sigma$ satisfies $\Phi(\Sigma(x))=(x,\alpha^?,\dots,\alpha^?)$ 
for all $x\in\Gamma$, and thus clearly induces a monomorphism of $G$. Hence
we obtain the $L$-presentation of Theorem \ref{gFG} for $G$.

\subsection{Generalized Gupta-Sidki groups}
\label{appC}

Assume now that $p$ is an odd prime, and consider the following group 
$G=\langle a,t\rangle$: its action on $X^*$ is specified by
\[\Psi(a,x)=(x+1,1),\qquad\Psi(t,0)=(0,t),\quad\Psi(r,x)=(x,a^x)\mbox{ if }x>0.\]
If $p=3$, this is the original Gupta-Sidki group, which was shown
in~\cite{GSi83} to be an infinite, finitely generated,
$3$-torsion group.

With a similar notation as above, we consider
$\Gamma=\langle\alpha,\tau|\alpha^p,\tau^p\rangle$, the normal closure
$\Delta$ of $\tau$, and the map $\Phi:\Delta\to\Gamma^p$ defined by
\[\Phi(\tau^{\alpha^i})=(\dots,\alpha^{p-1},\tau,\alpha,\alpha^2,\dots)\mbox{ with the $\tau$ at position $i$.}\]
In the group
\[\Gamma^p=\Pi=\langle\alpha_1,\dots,\alpha_p,\tau_1,\dots,\tau_p|\alpha_i^p,\tau_i^p,[\alpha_i,\alpha_j],[\alpha_i,\tau_j],[\tau_i,\tau_j]\mbox{ for }i\neq j\rangle,\]
we consider now the subgroup $\Phi(\Delta)=\langle\sigma_i:=\tau_i\alpha_{i+1}\dots\alpha_{i+k}^k\dots\alpha_{i-1}^{-1}\rangle$.
We rewrite the presentation of $\Pi$ as
\[\Pi=\langle\alpha_1,\dots,\alpha_p,\sigma_1,\dots,\sigma_p|\alpha_i^p,\sigma_i^p,[\alpha_i,\alpha_j],[\alpha_i,\sigma_j],[\sigma_i\alpha_j^{j-i},\sigma_j\alpha_i^{i-j}]\mbox{ for }i\neq j\rangle.\]

We choose as Schreier transversal all $p^p$ elements
$\alpha_1^{n_1}\dots\alpha_p^{n_p}$. The Schreier generating set
easily reduces to $\{\sigma_{i,n}:=\sigma_i^{\alpha_i^{n_i}}\}$. The
Schreier relations become
$\sigma_{i,m+i}^{-1}\sigma_{j,n+i}^{-1}\sigma_{i,m+j}\sigma_{j,n+j}$.
Furthermore, an easy calculation gives

\begin{equation} \label{eq:gs}
  \left[\sigma_i^{(j-k)e}\sigma_j^{(k-i)e},\sigma_k^{(i-j)e}
  \sigma_i^{(j-k)e}\right]=\sigma_{i,(j-i)(i-k)e}^{-2(j-k)e}\sigma_i^{2(j-k)e}.
\end{equation}

For all $\ell>0$, we may choose arbitrarily $j,k$ such that $i,j,k$
are all distinct and $(j-i)(i-k)/2(j-k)\equiv\ell\pmod p$, and
use equation (\ref{eq:sigmail}) to express $\sigma_{i,\ell}$ in terms of
$\sigma_i,\sigma_j,\sigma_k$, namely

\begin{equation}\label{eq:sigmail}
  \sigma_{i,\ell} = \sigma_i\left[\sigma_i^{1/2}\sigma_j^{(k-i)/2(j-k)},
  \sigma_k^{(i-j)/2(j-k)}\sigma_i^{1/2}\right]^{-1}.
\end{equation}

Finally, we may also use equation (\ref{eq:sigmail}) to construct an
endomorphism $\Sigma$; we summarize:
\begin{theorem}
  The subgroup $D=\langle t\rangle^G$ of the Gupta-Sidki $p$-group
  admits a finite ascending $L$-presentation with generators
  $\sigma_1,\dots,\sigma_p$ generating a free group $\Delta$; iterated
  relations
  \[\RELS=\left\{\sigma_i^p;\,\sigma_{i,m+i}^{-1}\sigma_{j,n+i}^{-1}\sigma_{i,m+j}\sigma_{j,n+j}\right\};\]
  and an endomorphism $\Sigma:\Delta\to\Delta$, defined by
  \[\Sigma(\sigma_i)=\sigma_{1,i}
  \mbox{ as given in equation (\ref{eq:sigmail})}.\]
\end{theorem}
It is not possible to extend $\Sigma$ to an endomorphism of $\Gamma$.
However, the extension of a finitely $L$-presented group by a finite
group is again $L$-presented; in the present case, it is a simple matter, from
the $L$-presentation of $D$, to construct the split extension
$G=D\rtimes_\zeta\Z/p\Z$, in which the automorphism $\zeta$ of $D$
cyclically permutes the generators.

\def\cprime{$'$}


\vspace*{0.5cm}
\noindent
\begin{tabbing}
4444444454567890123451234545555555545\=45555555555555555555555555454444444444444444444\kill
Bettina Eick                       \> Ren$\acute{\rm e}$ Hartung \\
Institut Computational Mathematics \> Institut Computational Mathematics \\
University of Braunschweig         \> University of Braunschweig \\
38106 Braunschweig                 \> 38106 Braunschweig \\
Germany                            \> Germany \\
beick@tu-bs.de                     \> r.hartung@tu-bs.de \\
$\mbox{}$ \\
Laurent Bartholdi \\
Ecole Polytechnique Federale \\
CH-1015 Lausanne \\
Switzerland \\
laurent.bartholdi@epfl.ch
\end{tabbing}

\noindent
June 21, 2007

\end{document}